\documentclass[reqno,english]{amsart}

\usepackage[foot]{amsaddr}
\numberwithin{equation}{section}
\usepackage{amsmath,amsfonts,mathtools, upgreek, amssymb,graphicx,amsthm,enumerate, url, bm, bbm}
\usepackage{stmaryrd,comment,paralist,mathrsfs,booktabs,tabularx,xifthen,xcolor,tikz,setspace}
\usetikzlibrary{decorations.pathmorphing,patterns,shapes,calc,decorations}
\usetikzlibrary{decorations.pathreplacing}
\usepackage[letterpaper]{geometry}
\usepackage[colorinlistoftodos]{todonotes}
\usepackage[colorlinks=true]{hyperref}
\usepackage[noadjust]{cite}
\usepackage[capitalize]{cleveref}

\usepackage{dja_macros} 

\title[sharp thresholds for integer feasibility]{Zero-One laws for random feasibility problems}
\author{Dylan J. Altschuler}
\address{D.J.\ Altschuler\hfill\break  Department of Mathematical Sciences \\ Carnegie Mellon University}
\email{daltschu@andrew.cmu.edu}

\begin{document}
\maketitle

\begin{abstract}
We introduce a general random model of a combinatorial optimization problem with geometric structure that encapsulates both linear programming and integer linear programming. Let $Q$ be a bounded set called the feasible set, $E$ be an arbitrary set called the constraint set, and $A$ be a random linear transform. We define and study the $\ell^q$-\textit{margin},
\[
    \M_q := d_{\ell^q}\pa{AQ, E}\,.
\]
The margin quantifies the feasibility of finding $y \in AQ$ satisfying the constraint $y \in E$.
Our contribution is to establish strong concentration of the $\ell^q$-margin for any $q \in (2,\infty]$, assuming only that $E$ has permutation symmetry. The case of $q = \infty$ is of particular interest in applications---specifically to combinatorial ``balancing'' problems---and is markedly out of the reach of the classical isoperimetric and concentration-of-measure tools that suffice for $q \le 2$.

Generality is a key feature of this result: we assume permutation symmetry of the constraint set and nothing else. This allows us to encode many optimization problems in terms of the margin, including random versions of: the closest vector problem, integer linear feasibility, perceptron-type problems, $\ell^q$-combinatorial discrepancy for $2 \le q \le \infty$, and matrix balancing. Concentration of the margin implies a host of new sharp threshold results in these models, and also greatly simplifies and extends some key known results. 
\end{abstract}

\section{Introduction}
Balancing covariates in experimental design, sparsifying a graph, and training a single-layer neural net are seemingly disparate combinatorial optimization problems. Nonetheless, they as well as a wide range of other problem can be recast as generalizations of the Closest Vector Problem, a core problem in integer programming.
\begin{quote}\label{txt:orig}
    \textit{Find the $\ell^q$-closest point of a set $Q$ to another set $E$.} 
\end{quote}
This is highly non-trivial for general $Q$ and $E$; crucially, properties like convexity are not assumed. A natural random model is to take a random linear transformation of either $Q$ or $E$.  

\begin{definition}[Random Feasibility Problem]
    Let $Q \subset \RR^{N}$ and $E \subset \RR^{M}$ be sets called the feasible set and constraint set, respectively. Fix a matrix $A \in \RR^{M \times N}$ with independent standard normal entries. The $\ell^q$-margin $\M_q(A) := \M_{q,Q,E}(A)$ is defined as:
    \[
        \M_q(A) := \min_{\sigma \in Q} d_{\ell^q}\pa{A\sigma, E}\,,
    \]
    and the $(A,Q,E)$-feasibility problem is the task of determining if $\M_q$ is zero.
\end{definition}    
The word margin is chosen in analogy to terminology from the literature of perceptron models. The margin quantifies the ``distance to satisfiability'' for the following program: find $\sigma \in Q$ with $A\sigma \in E$. In this program, $A$ and $E$ encode a random set of constraints on $Q$, and the margin $\M_q(A)$ measures the least the constraints can be violated by optimizing over $\sigma \in Q$. In particular, the margin is zero if and only if this program is feasible. The exponent $q$ controls how heavily the largest violations are penalized.

In the special case of the feasible set $Q$ being a subset of the integer lattice and the constraint set $E$ being a rectangle $[-\infty, b_1] \times \dots \dots [-\infty, b_{M}]$ for some vector $b \in \RR^{M}$, we recover the canonical form of random integer linear feasibility. Namely, the margin is zero if and only if there exists $\sigma \in Q$ satisfying
\[
    (A\sigma)_i \le b_i\,, \quad\forall\, i \in [M]\,.
\]
If $Q$ is $\RR^N$ instead, we recover random linear programming. 

The technical contribution of this article is strong concentration bounds for $\M_q$ under the assumptions that $E$ has sufficient permutation symmetry and $Q$ is bounded. Concentration for the margin can be interpreted as a sharp threshold as follows. Define the $\ell^q$-expansion of the constraint set $E$ by 
\begin{equation}\label{eq:expansion}
    E_\delta := \cb{x \in \RR^{M} ~:~ d_q(x, E) \le \delta}\,.
\end{equation}

The $\ell^q$-margin is exactly the smallest $\delta$ so that $E_\delta \cap AQ$ is non-empty. Thus, if the margin has fluctuations on some vanishing scale, then as one expands $E$ with respect to $\ell^q$, the probability that it contains a point of $AQ$ will jump from zero to one in a vanishing window. 

\subsection{Main Results}
\begin{definition}
    Say a set $E \subset \RR^{M}$ has \textit{permutation symmetry} if for any permutation $\pi \in S_{M}$ and $x \in \RR^{M}$,
        \[
            (x_1, \dots, x_M) \in E \quad \text{if and only if} \quad \pa{x_{\pi(1)},\dots, x_{\pi(M)}} \in E\,.
        \]
\end{definition}
This is a combinatorial notion of regularity. Typical examples include Cartesian products $E := (E_0)^{M}$ or the $\ell^p$ unit ball for any $p$. 

\begin{theorem}[Main result: concentration of the margin]\label{thm:margin} There is a universal constant $C>0$ so the following holds. Let $Q \subset \RR^{N}$ be a subset of the Euclidean unit ball and $E \subset \RR^M$ have permutation symmetry. For $q \in [2,\infty]$,
\begin{equation}\label{eq:main}
        \Var{\M_q(A)} \le \frac{C}{1 + \pa{\frac{1}{2} - \frac{1}{q}} \log M} \,.
    \end{equation}
\end{theorem}
By homogeneity, the assumption that $Q$ is a subset of the unit ball could be replaced by multiplying the right-hand side of \cref{eq:main} by a factor of $\max_{\sigma \in Q} \|\sigma\|_2^2$. At least for $q = \infty$, our result cannot be improved without further assumptions. Letting $Q$ be a singleton on the Euclidean unit sphere and $E$ be the origin in $\RR^M$
makes this clear (see Chapter 5, section 6 of \cite{chat-sc}). However, in many of the $(A,Q,E)$-feasibility problems that appear in actual applications, it should be the case that $\Var{\M_q(A)} \le 1/\mathrm{poly}(M)$. \cref{thm:margin} should thus be seen as an important but purely qualitative improvement over the bound of $\Var{\M_q(A)} \le 1$ that is implied by the Gaussian Poincar\'e inequality (see details in the proof of \cref{thm:margin}). \\

The assumption of permutation symmetry on $E$ is a serious restriction. Unfortunately, it cannot be completely dropped: letting $Q$ be a singleton comprised of a unit vector and setting $E = \RR^{M-1} \times [-1,1]$, clearly the fluctuations of the $\M_\infty$ are order one. Nonetheless, it is possible to somewhat relax the permutation symmetry condition on $E$. The following extension allows for imposing several different types of constraints on a feasibility problem. 

\begin{theorem}[Block symmetry suffices]\label{cor:block}
    There is a universal constant $C>0$ so that the following holds. Let $Q \subset \RR^N$ be a subset of the Euclidean unit ball and let $E = E_1 \times \dots \times E_{k}$ where $E_i \subset \RR^{M_i}$ has permutation symmetry for each $i$ and $\sum_{i=1}^k M_i = M$. Abbreviate $m := \min_{i \in k} M_i$. For each $q \in [2,\infty]$,
    \begin{equation}
        \Var{\M_q(A)} \le C  \pa{1 + \frac{1}{2}\log \pa{\frac{m^{1-\frac{2}{q}}}{k } \vee 1}}^{-1}\,.
    \end{equation}
\end{theorem}

\begin{remark}\label{rem:feas-sets}
    Let us briefly highlight the generality of \cref{thm:margin,cor:block}. While there are permutation symmetry requirements on $E$, there is nearly complete freedom for $Q$. The feasible set can be discrete, continuous, or even a singleton. Canonical examples include the sphere, solid cube, discrete cube, bounded subset of the integer lattice, and arbitrary subsets (such as level-sets of an arbitrary function) of any of the previous examples. Our results hold without distinguishing between these situations.
\end{remark}

Our main technical tool is Talagrand's $L^1$-$L^2$ (Gaussian) hypercontractive inequality, given below as \cref{lemma:L1L2}. There is a long and rich history of this inequality being used to prove similar sharp thresholds in various settings. See the expositions of Kalai \cite{kalai-summary} and Chatterjee \cite{chat-sc} for a wealth of examples in the Boolean and Gaussian settings, respectively. In particular, our result is similar to the classical theorems of Friedgut and Bourgain \cite{fried} and Friedgut and Kalai \cite{fried-kalai} that establish a sharp threshold for any permutation-transitive monotone Boolean function. The assumption of permutation-transitivity is used in their theorems for the same purpose as in ours.

\subsection{Future Directions}
Some open problems that seem attractive but quite challenging:
\begin{itemize}
    \item[Q1.] Obtain polynomial improvements over the given rates in the case that $Q$ has large cardinality and is well-spread (i.e. the average inner product over all pairs of elements of $Q$ is bounded away from one). A natural place to start is the dynamical variance identity given in Lemma 2.1 of \cite{chat-sc}. 
    \item[Q2.] Explore the trade-off between how close $E$ is to being permutation symmetric and how well the margin concentrates.
    \item[Q3.] Extend these results to other disorders for the matrix $A$. Independent Boolean or Rademacher entries are natural and may be tractable. Independent columns would be quite interesting but potentially harder. 
\end{itemize}

\subsection{Notation}
The phrase ``almost every'' will always be with respect to the Gaussian or Lebesgue measure. Since they are mutually absolutely continuous, this is without ambiguity. We say a function $F: \RR^{d} \to \RR$ is $(L,q)$-Lipschitz if $|F(x)-F(y)| \le L \|x - y\|_{q}$ for all $x$ and $y$. The set $E$ will always be assumed closed. For a vector $v$, let $|v|$ denote the result of applying the absolute value function entry-wise to $v$.

\section{Applications and Previous Works}

\subsection{Random Integer Programming}
Average-case linear programming and integer linear programming have been the subject of intense study over the past few decades, in large part due to the enormous gap between worst-case guarantees and average-case empirical performance of (integer) linear programming algorithms \cite{roughgarden-summary, spielman-teng, dadush-smooth}. Much of the work on integer programming focuses on understanding integrality gaps. Both the random setting \cite{dyer-frieze, chand-vemp, gap1, gap2} and deterministic setting have been extensively considered. (The latter is too rich to introduce here; we refer the reader to \cite{dadush-thesis}.) The study of integrality gaps often utilizes the notions of ``sensitivity'' and ``proximity,'' which quantify the distance between the vertices and lattice points contained in a polytope \cite{prox,sensitivity}. 

Much more closely related are the Shortest Vector Problem (SVP) and Closest Vector Problem (CVP) \cite{dadush-thesis, cvp}. The CVP asks: given a lattice and a target vector $v$, find $w$ in the lattice minimizing $\|v-w\|_q$. The SVP asks the same with $v = 0$ and the origin removed from the lattice. The author is not aware of work prior to this article on random versions of these problems.

\subsection{Combinatorial Discrepancy}
Combinatorial discrepancy arises as a fundamental quantity in a plethora of fields including combinatorics, geometry, optimization, information theory, and experimental design \cite{Spe94, chaz, matou}. For an $M \times N$ matrix $A$, the combinatorial discrepancy $\disc(A)$ is given by
\[
    \disc(A) := \min_{\sigma \in \frac{1}{\sqrt {N}}\cb{-1,+1}^N} \|A\sigma\|_\infty\,.
\]
There are a number of outstanding open conjectures that seem out of reach of current tools, motivating much recent interest in random \cite{ unified, bansal-smooth2, hoberg-rothvoss, franks2020discrepancy,potukuchi2018discrepancy,bansal-meka, ezra-lovett, dja-jnw, ALS1, ALS2, PX, ss, gamarnik-sbp} and semi-random models \cite{rada-smooth, bansal-smooth1}. It is a general feature in random discrepancy that the second moment method will yield a coarse threshold \cite{APZ}, but not a sharp threshold (unless $M \ll N$ \cite{turner2020balancing}). The second moment method is generically difficult to improve upon. 

However, a recent series of breakthroughs overcame this technical hurdle and established extremely precise control in the square regime $M \asymp N$. Perkins and Xu \cite{PX} and Abbe, Li, and Sly \cite{ALS1} simultaneously obtained a sharp threshold for the discrepancy of Gaussian and Rademacher matrices, respectively, as well as strong control over the number of solutions. Subsequently, the discrepancy of a Gaussian matrix was shown by the author to concentrate within a $\cO{\log(N)/N}$ window \cite{dja-crit}; the logarithm has been removed by Sah and Sawhney \cite{ss}. The mentioned breakthroughs all are quite technically involved. Our main theorem recovers a very simple, direct, and different proof (albeit with a highly non-optimal rate) of the recently established sharp threshold for discrepancy.  

Additionally, we also obtain some truly new results in related settings. A natural generalization of combinatorial discrepancy, the $\ell^q$ discrepancy of a matrix is given by 
\[
    \disc_q(A) := \min_{\sigma \in \cb{-1,+1}^N} \|A\sigma\|_q\,.
\]
While $\ell^q$ discrepancy has certainly been extensively studied for deterministic matrices (see e.g. \cite{colorful, anynorm} for a list of references; the literature dates back to at least a question of Dvoretsky in 1960, and is too vast to introduce here), to the best of our knowledge nothing is known about the random setting. Similarly to the $q  =\infty$ setting, it is natural to expect the second moment method to yield a coarse threshold. Any such future result can be automatically upgraded to a sharp threshold by applying our main result: 

\begin{theorem}[Sharp threshold for $\ell^q$ discrepancy]\label{prop:disc-lb}
    Let $A \in \RR^{N \times N}$ with independent standard normal entries and $2 \le q \le \infty$. Then:
    \[
        \frac{\E{\disc_q(A)}}{\sqrt{\Var{\disc_q(A)}}} \ge \cOm{N^{\frac{1}{q}}\sqrt{ 1 + \pa{\frac{1}{2} - \frac{1}{q}}\log N} }\,.
    \]
\end{theorem}

\begin{proof}
    The variance of $\disc_q(A)$ can be upper-bounded by applying \cref{thm:margin} with $E = \cb{0}$ and $Q $ the normalized discrete cube $N^{-1/2}\cb{-1,+1}^N$. We now turn to a lower bound on the expectation of $\M_q(A) := \M_q(A,Q,E)$. Since the margin is non-negative, it is enough to prove a high probability lower bound of order $N^{1/q}$ on $\M_q(A)$. By taking a union bound over the $2^N$ choices of $\sigma \in Q$, it further suffices to just show uniformly for all $\sigma \in Q$ that $\PP{\|A\sigma\|_q < cN^{1/q}} \le 2^{-2N}$ for a positive constant $c$. 
    
    Fix $\sigma \in Q$ and define $Y:=A\sigma$. Note that $Y$ is distributed as a standard normal random variable in $\RR^{M}$. Let $\mathcal{G}$ denote the event that at least $N/2$ entries of $Y$ are greater than $\del$ in absolute value for some small constant $\del >0$ fixed later. On the event $\mathcal{G}$, we have $\|A\sigma\|_q \ge c N^{1/q}$ where $c = \del\, 2^{-1/q}$. Performing a union bound over all possible locations of the smallest $N/2$ entries of $Y$ in absolute value,
    \[
        \PP{\|A\sigma\|_q < cN^{1/q}}\le \PP{\mathcal{G}^c} \le \binom{N}{N/2} \PP{|Y_1| \le \del} \le 2^N (2\del)^{N/2} \le 2^{-2N}\,,
    \]
    where the last inequality follows by taking $\del$ sufficiently small. The theorem follows.
\end{proof}

The result and proof remain essentially unchanged if we optimize discrepancy over any of the feasible sets given in \cref{rem:feas-sets} rather than taking $Q$ as the discrete hypercube. For example, another natural problem for which, to the best of our knowledge, there are no published results is the relaxation of random discrepancy to the sphere rather than the discrete cube. This is analogous to the relaxation of ``Ising'' spin glasses to ``spherical'' spin glasses in statistical physics. 

\begin{theorem}[Sharp threshold for the symmetric spherical perceptron]
For any $q \in [2,\infty]$ and $\alpha \in (0,\infty)$, let $A \in \RR^{\alpha N \times N}$ have iid standard normal entries.
\begin{equation*}
    \Var{\min_{\sigma \in S^{N-1}} \|A\sigma\|_q)} = \cO{\frac{1}{1 + \pa{\frac{1}{2}-\frac{1}{q}}\log N}}\,.
\end{equation*}
\end{theorem}
It seems reasonable to expect the second moment method to yield a coarse threshold in this setting as well---at least in some regimes of $\alpha$---which can thus be automatically promoted to a sharp threshold.    \\

We also note a related work of Minzer, Sah, and Sawhney \cite{minzer2023perfectly} that appeared during the writing of this article. They use the Boolean version of Talagrand's $L^1$-$L^2$ inequality to establish a sharp threshold in the ``perfectly friendly bisection'' problem. They note that their methods can be used to show concentration of random $\ell^\infty$ discrepancy with Bernoulli disorder. 

\subsection{Perceptron Models}
The perceptron is an enduring model of the classification power of a single-layer neural net. Despite being studied since the 1960's, the perceptron remains a fascinating yet stubborn source of open problems. The binary perceptron problem asks whether $\alpha N$ independent standard Gaussian points in $\RR^N$ with arbitrary binary labels can be perfectly separated with margin $K$, using a hyperplane. Two restrictions are imposed on the choice of hyperplane: it must pass through the origin and, motivated by practical considerations, the unit normal vector to the hyperplane be chosen from the (scaled) discrete cube. This problem turns out to be exactly the $(A,Q,E)$-feasibility problem with parameters:
\begin{equation}\label{eq:perceptron}
    q = \infty\,, \quad Q = N^{-1/2}\cb{-1,+1}^N \,, \quad E = [K,\infty)^{\alpha N}\,.
\end{equation}
Exploiting the fact that $E$ is a Cartesian power of a one-dimensional set, $\alpha$ can be varied while $K$ is held fixed to create a natural sequence of nested problems. The \textit{capacity} is the largest $\alpha$ for which the corresponding perceptron problem is feasible. More precisely, the capacity $\ac$ is the random variable:
\[
    \max\cb{\alpha : \exists \,\sigma \in Q,~ \innerprod{A_i}{\sigma} > 0, \quad \forall i \in [\alpha N]}\,,
\]
where $(A_i)_i$ is a collection of independent standard normal vectors in $\RR^N$. Concentration of $\ac$ was shown somewhat recently in two impressive and technically involved works \cite{nakajima-sun, xu}. Xu used the Fourier-analytic pseudo-junta theorem of Hatami to show concentration of $\ac$, establishing the first sharp threshold result for the binary perceptron \cite{xu}. Subsequently, Nakajima and Sun established a sharp threshold for a wide variety of related models \cite{nakajima-sun}; their methods extend some prior work of Talagrand \cite{tal-selfave,tal-mf2}. The generalization studied by Nakajima and Sun corresponds to letting $E:= (E_0)^{\alpha N}$, where $E_0$ can be any set that satisfies a mild structural assumption\footnote{See Assumption 1.2 and Theorem 1.2 of \cite{nakajima-sun} for details.}. In our notation, previous work can be summarized as:

\begin{theorem}[Sharp threshold for the capacity of the generalized perceptron \cite{nakajima-sun,xu}]\label{thm:prev-perceptron}
Let $q = \infty$, $Q = N^{-1/2}\cb{-1,+1}^{N}$ and $E_{\alpha} := (E_0)^{\alpha N}$. Then there exists some sequence $a_c := a_c(N, E_0)$ such that: for any $\eps > 0$ and $N$ sufficiently large, the $(A,Q,E_\alpha)$-feasibility problem is satisfiable with high probability for $\alpha < a_c - \eps$ and unsatisfiable with high probability for $\alpha > a_c  + \eps$. 

\end{theorem}

In the context of the binary perceptron problem posed in \cref{eq:perceptron}, the above theorem fixes $K$ and studies the 
largest feasible value of $\alpha$ (encoded by the capacity). We take a dual approach of fixing $\alpha$ and studying the largest feasible value of $K$ (encoded by the margin). These two views have concrete interpretations in terms of the binary classification problem introduced above. The former asks about the maximum number of points separable by a fixed margin, while the latter asks about the largest margin with which a fixed number of labeled points can be separated by. Both points of view are highly natural, and concentration of either objective has self-evident value. 

As a consequence of \cref{thm:margin}, we readily obtain concentration for the margin, yielding a sharp threshold in $K$ that is analogous to the sharp threshold in $\alpha$ provided by \cref{thm:prev-perceptron}. Recalling the notation $E_{\delta,q}$ for the $\ell^q$- expansion of the set $E$ defined in \cref{eq:expansion}, we have:

\begin{theorem}[Sharp threshold for the margin of the generalized perceptron]\label{thm:new-perceptron}
Let $ 2 < q \le \infty$, $Q \subset \RR^{N}$ be a subset of the Euclidean ball, and $E \subset \RR^{\alpha N}$ be permutation invariant. Then there exists some sequence $K_c := K_c(N)$ such that: for any $\eps > 0$ and $N$ sufficiently large, the $(A,Q,E_{K,q})$-feasibility problem is unsatisfiable with high probability for $K < K_c - \eps$ and satisfiable with high probability for $K > K_c  + \eps$. 
\end{theorem}

A sharp threshold for the margin should be viewed as complementary to rather than comparable with a sharp threshold for the capacity. Despite the close relation of the margin and the capacity, a sharp threshold for the margin and a sharp threshold for the capacity are difficult to directly compare outside of special situations (e.g. Lemma 3.4 of \cite{dja-crit}). What \textit{can} be compared, however, is the settings in which a sharp threshold for the capacity or
the margin can be established. \cref{thm:new-perceptron} appreciably expands the setting considered in  \cref{thm:prev-perceptron}. 

A primary feature of \cref{thm:new-perceptron} is that it treats \textit{any} bounded feasible set $Q$ rather than just the discrete hypercube. The choices of $Q$ considered in \cref{rem:feas-sets} yield sharp threshold results in a number of previously inaccessible models. We highlight one such example: consider the setting of \cref{eq:perceptron} with the change of: $Q = S^{N-1}$. This corresponds to the so-called spherical perceptron. While the positive spherical perceptron is already well-understood \cite{spherical-perceptron, tal-mf1, tal-mf2}, to the best of our knowledge a sharp threshold result was not previously known for the negative spherical perceptron. (``Positive'' and ``negative'' refer to parameter regimes of $\alpha$ for which $K_c$ is positive or negative, respectively. The negative spherical perceptron is conjectured to exhibit a property called ``full replica symmetry breaking'', which makes analysis difficult.) 

Additionally, \cref{thm:new-perceptron} can handle constraint sets $E$ that have permutation symmetry rather than product structure, and it can be used to study the $\ell^q$-margin for $q \neq \infty$. For example, in the context of the binary classification problems introduced above, $\M_q(A)$ corresponds to the performance of linear classifiers under an $\ell^q$ loss.\\

We conclude our discussion of perceptron models by mentioning a final generalization. \cref{cor:block}, the extension of our main theorem to block-symmetric constraint sets, immediately yields a sharp threshold for the margin of perceptron problems with ``random labels,'' such as the model raised in \cite{randlabel}. This corresponds to a feasibility problem with constraint set $E = \pa{E_0}^{M_1} \times (E_0^c)^{M - M_1}$ and $M_1 \asymp (M - M_1)$. 

\subsection{Matrix Balancing}
The ``Matrix Spencer'' conjecture of R. Meka asks the following: given $N$ symmetric matrices $A_1,\dots,A_N$ of dimension $d \times d$, each of operator norm at most one, determine the feasibility of finding $\sigma \in \frac{1}{\sqrt{N}}\cb{-1,1}^N$ such that 
\[
    \left\|\sum_i \sigma_i A_i\right\|_{\mathrm{op}} \le 1\,.
\]
This problem has interesting applications to quantum comunication complexity \cite{hopkins} and graph sparsification \cite{sparse1, sparse2}. There has been some recent progress \cite{4dev, polylog, hopkins, mirror} in the low-rank setting. Very recently, the random version of this problem has been considered; a lower bound by the first moment method is given in Theorem 1.13 of \cite{tim-mtx}. Due to the extremely strong lower-tail concentration of the operator norm of a GOE, the main regime of interest is $N \asymp d^2$. Only when $N$ is at least this large does optimizing over $\sigma$ allow $\|\sum_i\sigma_iA_i\|$ to be changed to first order; conversely, if $N$ is much larger, the discrepancy will be vanishing \cite{priv-com}.

\begin{theorem}[Sharp threshold for matrix balancing]\label{thm:mtx-spencer}
    Let $A_1,\dots,A_N$ be $d\times d$ GOE matrices, and define
    \[
        \disc(A_1,\dots,A_N) :=  \min_{\sigma \in N^{-1/2}\cb{-1,+1}^N} \frac{1}{\sqrt{d}}\left\|\sum \sigma_i A_i \right\|_{\mathrm{op}}\,.
    \]
    Then, for $N \lesssim d^2$,
    \begin{equation}
        \frac{\E{\disc(A_1,\dots,A_N)}}{\sqrt{\Var{\disc(A_1,\dots,A_N)}}} \ge \cOm{\sqrt{d}~} \,.
    \end{equation}
\end{theorem}
The proof is deferred to the end of the next section. Matrix balancing can be seen as an integer feasibility problem as follows: flatten each matrix $A_i$ into a $d^2 \times 1$ column vector. And let $A = (A_1,\dots,A_N)$. Let the constraint set $E \subset \RR^{d^2}$ be the flattening of the $d\times d$ operator norm ball, and let $Q$ be the discrete cube $N^{-1/2}\cb{-1,+1}^N$. However, matrix balancing is not easily quantified in terms of the $\ell^q$-margin and does not quite fit nicely into the framework of \cref{thm:margin} for two reasons. First, $E$ does not have much permutation symmetry since the operator norm of a matrix is only invariant under permuting rows or columns, not arbitrary entries. Second, \cref{thm:margin} would give results on the distance to $E$ in terms of entry-wise norms of the matrix $M$, which is not quite right. Nonetheless, the ideas behind the proof of \cref{thm:margin} are extremely general and easy modifications suffice here.

\section{Preliminaries}
Before proving our main theorems, we introduce our main tools as well as some technical observations that will help us apply these tools. We begin with the celebrated Gaussian Poincar\'e inequality.
\begin{lemma}[Gaussian Poincar\'e]\label{prop:g-pci}
    Let $n$ be finite, $f: \RR^{n} \to \RR$ be an absolutely continuous function, and $\gamma^n$ denote the standard Gaussian measure on $\RR^n$. 
    \begin{equation*}
        \Var{f} \le \E[\gamma^n]{\|\nabla f\|_{2}^2 }\,.
    \end{equation*}
\end{lemma}

The Gaussian Poincar\'e inequality is often quite useful, but sometimes fails to give optimal rates. In his influential monograph \cite{chat-sc}, Chatterjee gives the name ``superconcentration'' to the variance of a random variable being far smaller than the Poincar\'e inequality implies. Chatterjee also shows the equivalence of superconcentration to the fascinating properties of ``chaos'' and ``multiple valleys.'' It remains a major open problem in this area to establish general methods for providing polynomial improvements over the Gaussian Poincar\'e inequality. However, there is a general tool for obtaining logarithmic improvements: 
\begin{lemma}[Talagrand's $L^1$-$L^2$ inequality; Theorem 5.1 of \cite{chat-sc}]\label{lemma:L1L2}
    There is some constant $C$ so that the following holds. Let $\gamma^n$ be the Gaussian measure on $\RR^n$ for any $n$, and $f: \RR^n \to \RR$ any absolutely continuous function. 
    \begin{equation}\label{eq:l1l2}
        \Var{f}  \le C \sum_{i=1}^n \|\partial_{i} f\|_{L^2(\gamma^n)}^2 \pa{ 1 + \log\pa{\frac{\|\partial_{i} f\|_{L^2(\gamma^n)}}{\|\partial_{i} f\|_{L^1(\gamma^n)}}} }^{-1}\,.
    \end{equation}
    In particular, 
    \begin{equation}\label{eq:l1l2-jensen}
        \Var{f}  \le C \pa{\sum_{i=1}^n \|\partial_{i} f\|_{L^2(\gamma^n)}^2 }\pa{ 1 + \frac{1}{2}\log\pa{\frac{\sum_i\|\partial_{i} f\|_{L^2(\gamma^n)}^2}{\sum_i\|\partial_{i} f\|_{L^1(\gamma^n)}^2}} }^{-1}\,.
    \end{equation}
\end{lemma}
The usual formulation of Talagrand's inequality is \cref{eq:l1l2}. Here, \cref{eq:l1l2-jensen} will be more convenient for us. It readily follows from \cref{eq:l1l2} by applying Jensen's inequality to the function $g(x) = (1 + \log(x)/2)^{-1}$, which is concave on $(0,1)$. The details can be found in the proof of Theorem 5.4 of \cite{chat-sc}.\\

Talagrand's inequality is based on hypercontractivity---a way of quantifying the extreme smoothing properties of the heat flow. Roughly speaking, the Poincar\'e inequality can fail to yield optimal rates because it forces a heavy quadratic penalty on inputs for which $f$ has large derivative, even if the Gaussian measure assigns little mass to these locations. Hypercontractivity allows for mitigation of this penalty. \\

We plan to apply Talagrand's inequality to the variance of the margin. This requires differentiating the margin. Since the margin is defined in terms of a distance between two sets, this will consist of two tasks. First, interchanging a derivative and an infimum. Second, differentiating the $\ell^q$ distance. The former will be accomplished by the classical ``envelope theorem'' of Milgrom and Segal (Theorem 1 in \cite{milgrom-segal}). 

\begin{lemma}[Envelope Theorem]
\label{thm:envelope}
    Let $f:X \times [0,1] \to \RR$, where $X$ is a subset of $\RR^{n}$. Define the value function 
    \[
        V(t) := \sup_{x \in X} f(x,t)\,.
    \]
    For any $t \in (0,1)$ and any $x^* \in \arg\max_{x \in X} f(x,t)$, if $V'(t)$ and the partial derivative $f_t(x^*,t)$ both exist, 
    \[
        V'(t) = f_t(x^*,t)\,.
    \]
\end{lemma}

We next collect some regularity results on $\ell^q$ distances for the sake of differentiating $\ell^q$ distances later. In what follows, let $E \subset \RR^{m}$ be an arbitrary closed set and define
\[
    f_q(x) := d_q(x,E)\,.
\]
Recall our notation that a function $F: \RR^{M} \to \RR$ is called $(L,q)$-Lipschitz if it is Lipschitz continuous with constant $L$ with respect to $\ell^q$ perturbations to the input.
\begin{proposition}\label{prop:lip}
    Fix $q \in [2,\infty)$ and a non-empty, closed set $E \subset \RR^{M}$. Let $f_q(x) := d_q(x,E)$ be the $\ell^q$ distance between $x$ and $E$. The following hold:
    \begin{itemize}
        \item $f_q$ is $(1,q)$-Lipschitz continuous everywhere.
        \item $f_q$ is absolutely continuous everywhere and differentiable for Lebesgue a.e. $x$. 
        \item Let $x$ be a point of differentiability for $f_q(x)$. There is at least one $\ell^q$-projection of $x$ onto $E$. Letting $y$ denote such a projection and $v := x-y$, the vector $|v|$ does not depend on the choice of $y$. Additionally, for all $i \in [M]$ and Lebesgue a.e. $x$,
        \[
            \ba{\pa{\nabla f_q(x)}_i} = \frac{|v_i|^{q-1}}{\|v\|_{q}^{q-1}}\,.
        \]
    \end{itemize} 
\end{proposition}
These are standard facts, but we give a proof for completeness.

\begin{proof}[Proof of \cref{prop:lip}]
    Let $z \in E$ be arbitrary and let $x$ and $y$ be outside $E$. Fix $q \ge 2$ and set $d := d_q$. By triangle inequality,
    \[
        d(x,E) \le d(x,z) \le d(x,y) + d(y,z)\,. 
    \]
    Taking an infimum over $z \in E$ on the far right-hand side yields $d(x,E) - d(y,E) \le d(x,y)$. Reversing the roles of $x$ and $y$ yields a symmetric bound, completing the proof of the first claim. \\

    By equivalence of norms (i.e. since $\ell_q \le \ell_2 \le \sqrt{N} \ell_q$ for any $q \in [2,\infty]$), we have that $f_q$ is Lipschitz continuous with respect to the Euclidean norm. This implies absolute continuity everywhere and thus differentiability almost everywhere by Rademacher's theorem. \\

    Let $x$ be a point of differentiability of $f_q$ and $y$ be an $\ell^q$-projection of $x$ onto $E$. That is, $y \in \arg\min_{y' \in E} \|x-y'\|_{\ell^q}$. At least one such $y$ exists by the assumption that $E$ is closed. By definition of the derivative, there exists a unique vector $u$ (namely, the gradient of $f_q$ at $x$) such that for $w$ sufficiently near $x$,
    \begin{equation}\label{eq:deriv}
        f_q(w) - f_q(x) = u \cdot (w-x) + \co{\|w-x\|}\,.
    \end{equation}
    Since we just showed $f_q$ is $(1,q)$-Lipschitz, it follows that $\|u\|_{q^*} \le 1$ where $q^*$ is the conjugate exponent to $q$. Letting $w = x + \eps (y - x)$, we obtain by homogeneity of the $\ell^q$ norm
    \begin{align*}
        |f_q(w) - f_q(x)| &= \eps \|y-x\|_q\,.
    \end{align*}
    Combining this with \cref{eq:deriv}, \begin{equation}\label{eq:holder-equal}
        |u \cdot (w-x)| = \eps\|y-x\|_{q} + \co{\eps}\,.
    \end{equation}
    On the other hand, because $\|u\|_{q^*} \le 1$, applying H\"older's inequality yields:
    \begin{equation}\label{eq:holder-ineq}
        |u \cdot (w-x)| \le \|w-x\|_{q}\|u\|_{q^*} = \eps \|y-x\|_{q} \|u\|_{q^*} \le \eps \|y-x\|_{q}\,. 
    \end{equation}
    But, \cref{eq:holder-equal} shows that the application of H\"older's inequality in \cref{eq:holder-ineq} actually holds with equality, up to some vanishing error. For $q < \infty$, it is a standard fact \cite{holder} that this implies
    \[
        |u_i| = \frac{|y-x|_i^{q-1}}{\|y-x\|_{q}^{q-1}}\,.
    \]
    Here, we have adopted the convention that the above is zero if $\|y-x\|_q = 0$. In summary, for any $\ell^q$ projection $y$ of $x$ onto $E$, we have explicitly computed the gradient $u$ of $f_q(x)$ in terms of $x$ and $y$ up to the sign of its components. If there exists another projection $y'$ that does not satisfy: $|y'-x| = |y-x|$ for all $i$, then the uniqueness of $u$ (and thus the differentiability of $f_q$ at $x$) would be contradicted. So, $|y-x|$ does not depend on the choice of $y$. This completes all claims.
\end{proof}

Although similar statements hold for $q = \infty$, there is a lack of uniqueness for optimizers of H\"older's inequality, making direct analysis more difficult. We sidestep this issue with the standard approach of simply taking $q$ large. Let us make this precise. 

\begin{proposition}\label{prop:var-perturb}
    Let $X$ and $Y$ be random variables with finite second moments. Then:
    \begin{equation}\label{eq:var-add}
        \Var {X+Y} \le 2\pa{\Var X + \Var Y }
    \end{equation}
    and
    \begin{equation}\label{eq:var-perturb}
        \Var X \le 2\pa{\E{|X-Y|^2} + \Var Y }
    \end{equation}
\end{proposition}

\begin{proof}
Both are trivial consequences of the inequality $(a+b)^2 \le 2(a^2+b^2)$, valid for all real $a$ and $b$.     
\end{proof}

\begin{proposition}\label{prop:infty-approx}
    For any $q \ge \log(M)^2$, we have
    \[
        \M_\infty = \M_q\mO{\frac{1}{\log M}}\,.
    \]
\end{proposition}
In particular, combining \cref{prop:var-perturb} and \cref{prop:infty-approx} will allow us to study the variance of $\M_\infty$ via the variance of $\M_{\log(M)^2}$, which is easier to differentiate due to \cref{prop:lip}. We omit the proof of \cref{prop:infty-approx} since it is an immediate application of the classical fact: 
\[
    \|\cdot\|_{\infty} = \|\cdot\|_{\log(M)^2} \mO{\frac{1}{\log M}}\,. 
\]
We conclude our preliminary discussion by illustrating the utility of the Poincar\'e inequality with a quick proof of  \cref{thm:mtx-spencer}, the sharp threshold for random matrix balancing problem. The following classical fact is needed.  

\begin{lemma}\label{lemma:op-grad}
Let $A$ be a symmetric matrix of full rank with distinct eigenvalues. Let $\lambda$ be the largest eigenvalue and let $u$ be the corresponding unique eigenvector of unit norm. 
\begin{equation}\label{eq:op-grad}
    \frac{\partial}{\partial A_{ij}}\lambda = u_iu_j
\end{equation}
\end{lemma}
\begin{proof}[Proof of \cref{lemma:op-grad}]
Implicitly differentiate the formula 
\[
    Au = \lambda u\,,
\]
and then left multiply by $u^t$. Using that $u^tdu = (d\|u\|_2^2)/2 = 0$, and then using symmetry of $A$ so that $u^tA = \lambda u^t$, we obtain the result. 
\end{proof}

\begin{proof}[Proof of \cref{thm:mtx-spencer}] Let $G = (A^{(1)},\dots,A^{(N)})$ be a collection of $N$ matrices---each of dimension $d \times d$ with iid standard Gaussian entries---flattened into a vector of length $Nd^2$. Let $Q := N^{-1/2}\cb{-1,+1}^N$. The operator norm of $\sum_i \sigma_iA^{(i)}$ for any fixed $\sigma \in Q$ is Lipschitz in $G$, and thus $\disc(G)$ is Lipschitz as well. By Rademacher's theorem, both quantities are then differentiable for almost every $G$. Applying \cref{thm:envelope} to compute the partial derivatives of $\disc(G)$ yields that for almost every $G$ there is a vector $\sigma^* := \sigma^*(G) \in Q$ with $\disc(G) = d^{-1/2}\|\sum_i \sigma^*_i A^{(i)}\|$ and:
\[
    \ba{\partial_{A^{(i)}_{jk}} \disc(G)} = \ba{\frac{1}{\sqrt{d}} \partial_{A^{(i)}_{jk}} \left\|\sum \sigma_i^* A^{(i)} \right\|_{\mathrm{op}} } = \frac{1}{\sqrt{d}}u_j^*u_k^*\sigma_i^*, \quad \forall\,(j,k) \in [d]^2\,.
\]
Here, $u^*$ denotes the unit eigenvector associated with the top eigenvalue of $\sum \sigma^*_i A^{(i)}$. Finally, applying \cref{prop:g-pci} to $\disc(G)$ and using that both $u^*$ and $\sigma^*$ are unit vectors, we obtain:
\begin{align*}
    \mathrm{Var}\pa{\disc(G)} \le \frac{1}{d}~ \E{\|\nabla \disc(G) \|_2^2} = \frac{1}{d}~\sum_{ijk} \pa{u^*_ju^*_k\sigma_i^*}^2 = \frac{1}{d}\|u^*\|_2^4  = \frac{1}{d}\,.
\end{align*}
This is far from sharp and can easily be improved, even polynomially so. But it suffices for a sharp threshold. This concludes the upper bound on the variance. The lower bound that $\E{\disc(G)} \ge \cOm{1}$ follows from a first moment method (see Theorem 1.13 of \cite{tim-mtx}; this is the only place that the assumption $N \lesssim d^2$ is needed). 
\end{proof}

\section{Concentration of the margin}
\begin{proof}[Proof of \cref{thm:margin}]
By \cref{prop:infty-approx} in conjunction with \cref{prop:var-perturb}, we may assume without loss of generality that $q \le \log(M)^2$. In particular, $q$ is finite (but possibly $M$-dependent). \\

Let us verify the differentiability of the margin. Define the function $g(A,\sigma) = d_q(A\sigma, E)$. For any matrix $B \in \RR^{M \times N}$, we have by triangle inequality:
\begin{align*}
    \ba{g(A,\sigma) -g(B,\sigma)} &\le d_q(A\sigma, B\sigma) \\
    &\le \|A-B\|_{2,q}\|\sigma\|_2 \\
    &\le \|A-B\|_{2,q}\,.
\end{align*}
The final inequality follows from the theorem assumption that the feasible set is bounded in $\ell^2(\RR^N)$. By equivalence of matrix norms, we have that $g$ is $\ell^2$-Lipschitz with some finite (possibly $N$-dependent) constant, uniformly for $\sigma \in Q$. By Rademacher's theorem, $\nabla_A g(A,\sigma)$ exists for almost every $A$.

Similarly, the margin $\M_q(A)$ is also Lipschitz and thus differentiable for almost every $A$. Indeed, by the triangle inequality:
\begin{equation}\label{eq:margin-diffable}
    |\M_q(A) - \M_q(B)| \le \sup_{\sigma \in Q} \|A\sigma - B\sigma\|_q \le \|A-B\|_{2,q}\,.
\end{equation}

Additionally, by \cref{prop:lip} and the chain rule, we have for each $\sigma \in Q$ and almost every $A$ that the derivative of $g(A,\sigma)$ exists and is given by:
\begin{equation}\label{eq:partial-G}
    \ba{\partial_{A_{ij}} g(A,\sigma)} = |\sigma_j| \frac{|v_i|^{q-1}}{\|v\|_{q}^{q-1}}\,, \quad \forall~i \in [M],~j \in [N]\,,
\end{equation}
where $z$ is an $\ell^q$-projection of $A\sigma$ onto $E$ and $v = A\sigma - z$. Recall that $|v|$ is well-defined (i.e. does not depend on the choice of $z$) by \cref{prop:lip}. And, as verified in \cref{eq:margin-diffable}, $\nabla \M_q(A)$ exists for almost every $A$. Thus, for almost every $A$ and any $\sigma^*$ with $g(A,\sigma^*) = \M(A)$, applying \cref{thm:envelope}  to compute each partial derivative of $\M_q(A)$ using \cref{eq:partial-G} yields:
\[
     \left|\nabla \M(A)\right| =\pa{ |\sigma_j^*| \frac{|v_i|^{q-1}}{\|v\|_{q}^{q-1}} }_{i \in [M],~j \in [N]}\,.
\]

We turn to bounding the variance of $\M_q$ using Talagrand's $L^1$-$L^2$ inequality. The quantities involved in \cref{eq:l1l2-jensen} of \cref{lemma:L1L2} are:
    \begin{align*}
        a_{ij}^2 &:= \E{\ba{\partial_{A_{ij}} \M_{q,E}(A) }}^2 \\
        b_{ij}^2 &:=\E{\ba{\partial_{A_{ij}} \M_{q,E}(A) }^2}\,.
    \end{align*}
By H\"older's inequality, the $q$ and $(q-1)$ norm of $v$ cannot be too far apart:
\begin{equation}\label{eq:norm-comp}
    \|v\|_q^{q-1} \le \|v\|_{q-1}^{q-1} \le \pa{M}^{\frac{1}{q}} \|v\|_{q}^{q-1}\,.
\end{equation}
Additionally, by permutation-invariance of the constraint set $E$, 
\begin{equation}\label{eq:perm-invar-a}
    a_{ij}^2 = \E{\ba{\sigma_j^*} \frac{|v_i|^{q-1}}{\|v\|_{q}^{q-1}}}^2 = \pa{\frac{1}{M}~\E{\ba{\sigma_j^*} \frac{\|v\|_{q-1}^{q-1}}{\|v\|_{q}^{q-1}}}}^2\,.
\end{equation}
Combining \cref{eq:norm-comp} and \cref{eq:perm-invar-a},
\begin{align}
    \sum_{ij} a_{ij}^2 &= \frac{1}{M}\sum_{j}~\E{ \ba{\sigma_j^*} \frac{\|v\|_{q-1}^{q-1}}{\|v\|_{q}^{q-1}}}^2 \le M^{\frac{2}{q} - 1}\sum_{j}~\E{ \ba{\sigma_j^*} }^2 \le M^{\frac{2}{q} - 1}\E{ \|\sigma^*\|_2^2} \le M^{\frac{2}{q} - 1}\,.
     \label{eq:ub-a}
\end{align}
The penultimate inequality holds by an application of Jensen's inequality, and the final inequality follows from the assumption that the feasible set is a subset of the Euclidean ball. By monotonicity of $\ell^p$ norms, we also have
\begin{align}
    \sum_{ij} b_{ij}^2 &= \sum_{j}~\E{ \ba{\sigma_j^*}^2 \pa{\frac{\|v\|_{2(q-1)}}{\|v\|_{q}} }^{2(q-1)} } \le \sum_{j}~\E{ \ba{\sigma_j^*}^2 } \le 1 \,.
     \label{eq:ub-b} 
\end{align}
Finally, it is readily checked that for any $a > 0$, the following is monotone increasing in $x$ on $[a^{-1},\infty)$:
\[
    f(x) = \frac{x}{1 + \frac{1}{2}\log(ax)}\,.
\]
By Jensen's inequality and \cref{eq:ub-b}, we have $\sum_{ij}a_{ij}^2 \le \sum_{ij}b_{ij}^2 \le 1$. So, (in order of inequalities below) applying \cref{eq:l1l2-jensen} of \cref{lemma:L1L2}, then monotonicity of $f$, and then \cref{eq:ub-a} yields:
\begin{align*}
    \Var{\M_q} \le  \frac{C\sum_{ij}b_{ij}^2}{1 + \frac{1}{2}\log\pa{ \frac{\sum_{ij}b_{ij}^2} {\sum_{ij}a_{ij}^2}  }} \le \frac{C}{1 + \frac{1}{2}\log\pa{ \frac{1} {\sum_{ij}a_{ij}^2}  }} \le \frac{C}{1 + \frac{1}{2}\log\pa{ M^{1-\frac{2}{q}}}} \,.
\end{align*}
\end{proof}

\begin{remark}
    The reason that it works to apply H\"older's inequality so crudely is the following dichotomy: we are already done by Poincar\'e if $\sum_{ij} b_{ij}^2 \ll 1$. On the other hand, if $\sum_{ij} b_{ij}^2 \approx 1$ then the $q$ and $2(q-1)$ norms of $v$ are not very far apart, so $v$ must be a highly structured vector, in which case the $q$ and $q-1$ norms of $v$ must also be comparable. Then $b^2 \gg a^2$ and Talagrand's inequality easily yields a logarithmic improvement. 
\end{remark}

\begin{proof}[Proof of \cref{cor:block}]
     We adopt the notation and proof of \cref{thm:margin}. The first section on differentiability of the margin remains unchanged. Partition $[M]$ into $I_1,\dots,I_k$ so that $|I_j| = M_j$ by setting $I_1 = [M_1]$, $I_2 = \pa{M_1+1,\dots,M_1 + M_2}$, and so on.  Let $w_1 := v_{I_1}$ be the $M_1$-dimensional vector induced by taking the subset of $v$ with indices in $I_1$. Define $w_j := v_{I_j}$ similarly. Then the concatenation $(w_1, \dots w_k)$ is simply $v$. Arguing as in the proof of \cref{thm:margin}, it is easy to check by permutation symmetry:
    \begin{align*}
        \sum_{i\in I_t}\sum_j a_{ij}^2 &=  \frac{1}{M_t}\sum_j \E{|\sigma^*_j|\pa{\frac{\|w^{(t)}\|_{(q-1)}}{\|v\|_{q}} }^{(q-1)} }^2 \\
        &\le M_t^{\frac{2}{q}-1}\sum_j \E{|\sigma^*_j|\pa{\frac{\|w^{(t)}\|_q}{\|v\|_q} }^{(q-1)}}^2 \\
        &\le M_t^{\frac{2}{q}-1}\,.
    \end{align*}
    Note that by Jensen's inequality $\sum_{ij} b_{ij}^2 \ge \sum_{ij} a_{ij}^2$ always. Thus, by \cref{eq:l1l2-jensen} of \cref{lemma:L1L2}, we have:
    \begin{align*}
        \Var{\M_q} &\le C \pa{\sum_{ij} b_{ij}^2} \pa{1 + \frac{1}{2}\log \pa{\frac{\sum_{ij} b_{ij}^2 }{ \sum_{ij} a_{ij}^2 }  \vee 1 } }^{-1}  \\
        &\le C \pa{1 + \frac{1}{2}\log \pa{\frac{1}{ \sum_{t=1}^k M_t^{\frac{2}{q}-1}  } \vee 1 } }^{-1} \\&\le C \pa{1 + \frac{1}{2}\log \pa{\frac{\pa{\min_{t\in [k]} M_t}^{1 - \frac{2}{q}}}{k} \vee 1 }}^{-1} \,.
    \end{align*}
    From the first to the second line we have used monotonicity in $x$, for any $a > 0$ and all $x \in \RR$, of the function 
    \[
        f(x) = \frac{x}{1 + \frac{1}{2}\log (ax \vee 1)}\,.
    \]
\end{proof}

\section*{Acknowledgments}
I am deeply grateful to Jonathan Niles-Weed for invaluable support, advice, and mentorship throughout all stages of this project. I benefited from useful and encouraging conversations with Paul Bourgade, Guy Bresler, Brice Huang, Mark Sellke, Joel Spencer, Nike Sun, Konstantin Tikhomirov, and Will Perkins during the long exploratory phase of this project. I also thank the anonymous reviewers for helpful feedback that improved the accuracy of this work. Finally, the monograph \cite{chat-sc} of Chatterjee was an important source of inspiration for this work. This article is adapted from a chapter of the author's dissertation at NYU.

\bibliography{feas}       
\bibliographystyle{acm}

\end{document}